\documentclass{article}

\usepackage[latin2]{inputenc}
\usepackage{t1enc, anysize, enumerate, amsthm, amsmath, amssymb, epsfig}
\usepackage[unicode]{hyperref}
\newtheorem{theorem}{Theorem}[section]
\newtheorem{lemma}[theorem]{Lemma}

\newtheorem{remark}[theorem]{Remark}
\newtheorem{question}[theorem]{Question}

\newenvironment{definition}[1][Definition]{\begin{trivlist}
\item[\hskip \labelsep {\bfseries #1}]}{\end{trivlist}}
\newenvironment{conjecture}[1][Conjecture]{\begin{trivlist}
\item[\hskip \labelsep {\bfseries #1}]}{\end{trivlist}}
\newenvironment{notation}[1][Notation]{\begin{trivlist}
\item[\hskip \labelsep {\bfseries #1}]}{\end{trivlist}}

\begin{document}

\title{Density of 4-edge paths in graphs with fixed edge density}
\author{D\'{a}niel T. Nagy\footnote{E\"{o}tv\"{o}s Lor\'and University, Budapest. dani.t.nagy@gmail.com}}
\maketitle

\begin{abstract}
We investigate the number of 4-edge paths in graphs with a fixed number of vertices and edges. An asymptotically sharp upper bound is given to this quantity. The extremal construction is the quasi-star or the quasi-clique graph, depending on the edge density. An easy lower bound is also proved. This answer resembles the classic theorem of Ahlswede and Katona about the maximal number of 2-edge paths, and a recent theorem of Kenyon, Radin, Ren and Sadun about $k$-edge stars.
\end{abstract}

\section{Introduction}
The aim of this paper is to asymptotically determine the maximal and minimal number of 4-edge paths in graphs with fixed number of vertices and edges.

The first result of this kind is due to Ahlswede and Katona \cite{ahl}, who described the graphs with a fixed number of vertices and edges containing the maximal number of 2-edge paths. To state this result, we need some simple definitions.

\begin{definition}
The quasi-clique $C_n^e$ is a graph with $n$ vertices and $e$ edges, defined as follows. Take the unique representation
$$e=\binom{a}{b}+b~~~~~~0\le b<a,$$
connect the first $a$ vertices to each other, and connect the $a+1$-th vertex to the first $b$ vertices
\end{definition}

\begin{definition}
The quasi-star $S_n^e$ is a graph with $n$ vertices and $e$ edges, defined as follows. Take the unique representation
$$\binom{n}{2}-e=\binom{p}{2}+q,~~~~~~0\le q<p,$$
connect the first $n-p-1$ vertices with every vertex, and connect the $n-p$-th vertex with the first $n-q$ vertices.
\end{definition}

It is easy to see that $S_n^e$ is isomorphic to the complement of $S_n^{\binom{n}{2}-e}$.

\begin{notation}
The number of 2-edge paths in $C_n^e$ and $S_n^e$ is denoted by $C(n,e)$ and $S(n,e)$ respectively, while the number of $k$-edge stars is denoted by $C_k(n,e)$ and $S_k(n,e)$ respectively.
\end{notation}

\begin{theorem} {\bf (Ahlswede and Katona, 1978, \cite{ahl})} \label{ahlkat}
Let $G$ be a simple graph with $n$ vertices and $e$ edges. Then the number of 2-edge paths in $G$ is at most $\max(C(n,e),~ S(n,e))$.

Furthermore,
$$\max(C(n,e),~ S(n,e))=
\begin{cases}
S(n,e)~~~~~~\textrm{if}~~ 0\le e\le\frac{1}{2}\binom{n}{2}-\frac{n}{2}, \\
C(n,e)~~~~~~\textrm{if}~~ \frac{1}{2}\binom{n}{2}+\frac{n}{2}\le e \le \binom{n}{2}.
\end{cases}$$
\end{theorem}

Roughly speaking, this theorem states that if the edge density if smaller than $\frac{1}{2}$, then the quasi-star is the extremal example, while for higher edge densities the quasi-clique becomes extremal. (The transition between the two cases happens in a nontrivial way.)

Recently, Kenyon, Radin, Ren and Sadun proved a similar result for $k$-edge stars, using the notion of graphons. Translating the result back to language of graphs, we get the following theorem:

\begin{theorem} {\bf (Kenyon, Radin, Ren and Sadun, 2014, \cite{star})} \label{starthm}
Let $G$ be a simple graph with $n$ vertices and $e$ edges, and let $2\le k\le 30$. Then the number of k-edge stars in $G$ is at most $$\max(C_k(n,e),~ S_k(n,e))(1+O(e^{-\frac{1}{2}})).$$
\end{theorem}

The theorem is conjectured to hold for all values of $k$. (The only thing left to prove this, is a complicated extremal value problem.) Similarly to the case of the $2$-edge path, $C_k(n,e)<S_k(n,e)$ if the edge density is small, and $S_k(n,e)<C_k(n,e)$ if it is greater. The point of transition depends on $k$.

Now let us discuss three theorems with just one fixed parameter: the number edges. (So $n$ is not fixed.) We will start with a general theorem of Alon.

\begin{theorem} {\bf (Alon, 1981, \cite{alon1})} \label{alonthm}
Let $N(G,H)$ denote the number of subgraphs of $G$ that are isomorphic to $H$. Assume that $H$ is a single graph that has a spanning subgraph which is the vertex-disjoint union of edges and cycles. Then
$$N(G,H)\le (1+O(e^{-\frac{1}{2}}))N(C_n^e, H).$$
\end{theorem}

It means that for these graphs $H$, the asymptotically extremal example is always the quasi-clique. Note that this theorem can be applied in the case of fixed $n$ and $e$, since the extremal example provided by it is the quasi-clique. (No matter how many vertices we are given, we just have to construct a quasi-clique of $e$ edges.)

Also note that this theorem provides upper bounds for all graphs with a perfect matching, (for example all paths with an odd number of edges) and Hamiltonian graphs (for example complete graphs). In the case of the triangle graph $K_3$, the asymptotically best lower bound was proved by Razborov \cite{raz}.

The problem of finding the maximal number of 4-edge paths in graphs with $e$ edges (and an unlimited number of vertices) was solved by Bollob\'{a}s and Sarkar.

\begin{theorem} {\bf (Bollob\'{a}s and Sarkar, 2003, \cite{sarkar2})}
The number of 4-edge paths among graphs with $e$ edges is maximized by the graph that is obtained by taking the complete bipartite graph $K_{2,\lceil e/2\rceil}$, and deleting an edge if $e$ is odd.
\end{theorem}

Bollob\'as and Sarkar also proved asymptotic results for $2k$-edge paths. \cite{sarkar1} The extremal example in this case is the complete bipartite graph with $k$ vertices in one side. For $2k+1$-edge paths, the asymptotically extremal example is the quasi-clique. It follows from Theorem \ref{alonthm}, and is also proved in \cite{sarkar1}.

Alon had a conjecture for star-forests (vertex-disjoint union of stars), which was partially verified by F\"uredi.

\begin{conjecture} {\bf (Alon, 1986, \cite{alon2})}
Let $H$ be a star-forest. For any $e>0$, the graph maximizing the number of subgraphs isomorphic to $H$ among graphs with $e$ edges is a star-forest.
\end{conjecture}

\begin{theorem}{\bf (F\"uredi, 1992, \cite{furedi})}
Let $H$ be star-forest consisting of components with $a_1, a_2, \dots a_t$ edges. Assume that $a_i>\log_2(t+1)$ holds for all $1\le i\le t$. Let $e$ be sufficiently large. Then the graph maximizing the number of subgraphs isomorphic to $H$ among those with $e$ edges is a star-forest with $t$ components.
\end{theorem}

Considering the above results, investigating the number of the 4-edge paths seems to be the "natural" choice in the case of fixed $(n,e)$. In this paper, an asymptotic upper bound will be given to this quantity. Similarly to the case of $k$-edge stars, the asymptotically extremal graphs are the quasi-stars and the quasi-cliques. We will also prove an easy asymptotic lower bound.

\section{Proof of the main result}
\begin{theorem} \label{main}
Let $G$ be a simple graph with $n$ vertices and $e$ edges. Let $c=\frac{2e}{n^2}$. (Then $0\le c \le 1$.) Let $N$ denote the number of 4-edge-paths. Then
$$\frac{1}{2}c^4n^5(1-O(n^{-1}))\le N \le \frac{1}{2}\max((1-\sqrt{1-c})^2((c+1)\sqrt{1-c}+c),~c^\frac{5}{2})n^5.$$
\end{theorem}

\begin{proof}
Let $N'$ denote the number of the sequences $\{v_0, v_1, v_2, v_3, v_4\}$ where $v_i$ are (not necessarily different vertices) of $G$ and $v_{i-1}v_i\in E(G)$ for $i=1,\dots 4$. Here, we count every 4-edge path twice (there are two directions). We also count some walks of length 4 with repeated vertices. However, the number of such walks is only $O(n^4)$. Therefore $2N \le N'\le 2N+O(n^4)$, so it suffices to prove
$$c^4\le \frac{N'}{n^5} \le \max((1-\sqrt{1-c})^2((c+1)\sqrt{1-c}+c),~c^\frac{5}{2}).$$

Let us note that $\frac{N'}{n^5}$ is often referred to as the homomorphism density of the 4-edge path $P^4$ in $G$, and denoted by $t(P^4,G)$. (See \cite{lovaszbook} for an overview in the topic of graph homomorphisms.)

First, we prove the lower bound, which is much easier. If we want to select a 4-edge walk, we can start by choosing $v_2$, then $v_1$ and $v_3$ (we have to pick them from $N(v_2)$), and finally $v_0$ and $v_4$ ($\deg(v_1)$ and $\deg(v_3)$ possibilities). So we can write $N'$ as below, and estimate it by using twice that $\displaystyle\sum_{i=1}^m x_i^2\ge \frac{1}{n}\left(\displaystyle\sum_{i=1}^m x_i\right)^2$ holds for all real numbers.

$$N'=\sum_{v_2\in V(G)}\left(\sum_{v_i\in N(v_2)} \deg(v_i) \right)^2\ge \frac{1}{n}\left(\sum_{v_2\in V(G)} \left(\sum_{v_i\in N(v_2)} \deg(v_i)\right)\right)^2=$$

$$\frac{1}{n}\left(\sum_{v_i\in V(G)} \deg(v_i)^2 \right)^2\ge
\frac{1}{n}\left(\frac{1}{n}\left(\sum_{v_i\in V(G)} \deg(v_i)\right)^2 \right)^2=
\frac{1}{n}\left(\frac{1}{n}\left(cn^2\right)^2 \right)^2=c^4n^5.$$

Now we move on to the proof of the upper bound. Let codeg($v,w$) denote the number of common neighbours of the vertices $v$ and $w$. Note that
$$N'=\sum_{v_1, v_3\in V(G)} \deg(v_1)\deg(v_3)\textrm{codeg}(v_1,v_3),$$
since after fixing $v_1$ and $v_3$, we have $\deg(v_1)$ candidates for $v_0$, $\deg(v_3)$ candidates for $v_4$, and $\textrm{codeg}(v_1,v_3)$ candidates for $v_2$. Obviously, $\textrm{codeg}(v_1,v_3)\le \min(\deg(v_1), \deg(v_3))$, therefore
$$N'\le\sum_{v_1, v_3\in V(G)} \deg(v_1)\deg(v_3)\min(\deg(v_1), \deg(v_3)).$$

\begin{definition}
Let $G$ be a simple graph with $n$ vertices labeled $w_1, w_2, \dots w_n$. $A_G:[0,1)^2\rightarrow [0,1]$ is the function that is 1 on all rectangles $[\frac{i-1}{n}, \frac{i}{n})\times [\frac{j-1}{n}, \frac{j}{n})$ where $w_iw_j\in E(G)$, and 0 elsewhere.
\end{definition}

\begin{definition}
Let $A:=[0,1)^2\rightarrow [0,1)$ be an integrable function satisfying $A(x,y)=A(y,x)$ for all $0\le x,y<1$. Then for all $0\le x<1$ let
$$\ell(x)=\int_0^1 A(x,y) \, \mathrm{d}y$$
and let
$$S(A)=\int_0^1 \int_0^1 \ell(x)\ell(y)\min(\ell(x), \ell(y)) \, \mathrm{d}x \mathrm{d}y.$$
\end{definition}

Note that $A_G$ satisfies $\int_0^1 \int_0^1 A_G(x,y) \, \mathrm{d}x \mathrm{d}y=c$ and $A_G(x,y)=A_G(y,x)$. If $x\in [\frac{i-1}{n}, \frac{i}{n})$, then $\ell(x)=\frac{\deg(w_i)}{n}$, so
$$S(A_G)=\frac{1}{n^5}\sum_{1\le i,j \le n} \deg(w_i)\deg(w_j)\min(\deg(w_i), \deg(w_j))\ge \frac{N'}{n^5}.$$

\begin{definition}
Let $0\le c \le 1$. Then let $A_1(c):[0,1)^2\rightarrow [0,1]$ be the function satisfying $A_1(x,y)=1$ if $\min(x,y)<1-\sqrt{1-c}$ and $A_1(x,y)=0$ otherwise. Let $A_2(c):[0,1)^2\rightarrow [0,1]$ be the function satisfying $A_2(x,y)=1$ if $\max(x,y)<\sqrt{c}$ and $A_2(x,y)=0$ otherwise. (It is easy to see that $\int_0^1 \int_0^1 A_i(x,y) \, \mathrm{d}x \mathrm{d}y=c$ holds for $i=1,2$.) See Figure \ref{a1a2}.
\end{definition}

\begin{figure}[h]
\begin{center}
\includegraphics[scale = 0.5] {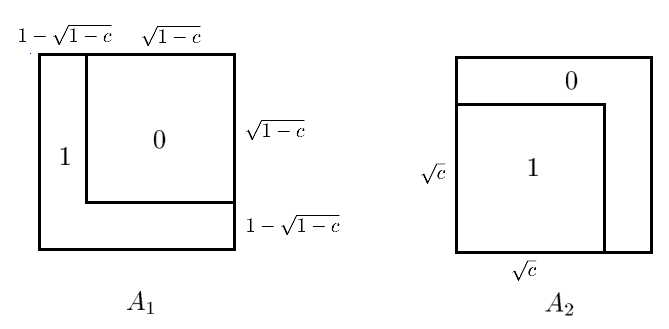}
\end{center}
\caption{The functions $A_1(c)$ and $A_2(c)$.}
\label{a1a2}
\end{figure}

Theorem \ref{main} will be an easy consequence of the following theorem.

\begin{theorem} \label{thm2}
Let $0\le c \le 1$ and $K\in\mathbb{N}^+$ fixed numbers. Assume that $A:=[0,1)^2\rightarrow [0,1)$ is a function satisfying $\int_0^1 \int_0^1 A(x,y) \, \mathrm{d}x \mathrm{d}y=c$ and $A(x,y)=A(y,x)$ for all $0\le x,y<1$, and that there are some numbers $0=q_0<q_1<q_2<\dots < q_K=1$ such that $A$ is constant on $[q_{i-1}, q_i)\times [q_{j-1}, q_j)$ for all $1\le i,j\le K$. Then
$$S(A)\le \max(S(A_1(c)),S(A_2(c)))=\max((1-\sqrt{1-c})^2((c+1)\sqrt{1-c}+c),~c^\frac{5}{2}).$$
\end{theorem}

\begin{proof}
We will use the following notations. $I_i=[q_{i-1},q_i)$, $t_i=q_i-q_{i-1}=|I_i|$, $\ell_i=\ell(x)$ for any $x\in I_i$, $A_{i,j}$ is the value of $A$ in the rectangle $I_i\times I_j$. We will refer to the sets of the form $[0,1)\times I_i$ and $I_i\times [0,1)$ as rows and columns respectively.

The function $S$ is continuous on a compact set defined by the conditions, so its maximum is attained for some $A$. Let $A$ be a function maximizing $S$, and let
$$T(A)=\int_0^1 \int_0^1 \int_0^1 \int_0^1 |A(x_1,y_1)-A(x_2,y_2)| \, \mathrm{d}x_1 \mathrm{d}y_1 \mathrm{d}x_2 \mathrm{d}y_2=\sum_{1\le a_1, a_2, b_1, b_2 \le K} t_{a_1}t_{b_1}t_{a_2}t_{b_2}|A_{a_1,b_1}-A_{a_2,b_2}|.$$

By a similar compactness argument, the minimum of $T$ is also attained for some $A$ (among those that maximize $S$). Such an $A$ can not have four rectangles $I_{i_1}\times I_{j_1}, I_{i_1}\times I_{j_2}, I_{i_2}\times I_{j_1}$ and $I_{i_2}\times I_{j_2}$ satisfying $A_{i_1,j_1}< A_{i_2,j_1}$ and $A_{i_1,j_2}> A_{i_2,j_2}$.

For some $\varepsilon>0$, replace the values $A_{i_1,j_1}, A_{i_2,j_1}, A_{i_1,j_2}$ and $A_{i_2,j_2}$ by $A_{i_1,j_1}+\frac{\varepsilon}{t_{i_1}t_{j_1}}, A_{i_2,j_1}-\frac{\varepsilon}{t_{i_2}t_{j_1}}, A_{i_1,j_2}-\frac{\varepsilon}{t_{i_1}t_{j_2}}$ and $A_{i_2,j_2}+\frac{\varepsilon}{t_{i_2}t_{j_2}}$ respectively. By choosing a small enough $\varepsilon$, the value of $A$ remains greater in $I_{i_2}\times I_{j_1}$ and $I_{i_1}\times I_{j_2}$ than in $I_{i_1}\times I_{j_1}$ and $I_{i_2}\times I_{j_2}$ respectively. Note that such a change does not change the values $\ell(x)$, therefore not changing $S(A)$. (To see that, take a line that intersects two of the four rectangles where the value of $A$ changes. It increases in one of them, while decreasing in the other one. This results in a 0 net change in the integral of $A$ over that line, since if one of the rectangles intersect the line in a segment $\lambda$ times as long as the other one, then its area is $\lambda$ times greater, so the change in the value of $A$ is $\lambda$ times smaller.)

Now we show that the value $T(A)$ decreases during this transformation. $T(A)$ is the sum of differences between the values $A_{i,j}$, weighted with the areas of these rectangles. Assume that the value of $A$ is greater in $r_1$ than in $r_2$ for two rectangles $r_1$ and $r_2$. If we decrease the value of $A$ in a rectangle $r_1$ with $\frac{\varepsilon}{Area(r_1)}$, and increase it in $r_2$ with $\frac{\varepsilon}{Area(r_2)}$ for a small enough $\varepsilon$, then $T(A)$ decreases. To see that, note that
$$Area(r_1)Area(r_2)|A_{r_1}-A_{r_2}|>
Area(r_1)Area(r_2)\left|A_{r_1}-\frac{\varepsilon}{Area(r_1)}-\left(A_{r_2}+\frac{\varepsilon}{Area(r_2)}\right)\right|$$
and for any rectangle $r_3\not\in\{r_1,r_2\}$
$$Area(r_1)Area(r_3)|A_{r_1}-A_{r_3}|+Area(r_2)Area(r_3)|A_{r_2}-A_{r_3}|\ge$$
$$Area(r_1)Area(r_3)\left|A_{r_1}-\frac{\varepsilon}{Area(r_1)}-A_{r_3}\right|+Area(r_2)Area(r_3)\left|A_{r_2}+\frac{\varepsilon}{Area(r_2)}-A_{r_3}\right|.$$
Applying this to $(r_1,r_2)=(I_{i_2}\times I_{j_1}, I_{i_1}\times I_{j_1})$ and $(r_1,r_2)=(I_{i_1}\times I_{j_2}, I_{i_2}\times I_{j_2})$ the desired result follows.

The symmetry of $A$ can be ruined by this transformation, but replacing $A(x,y)$ by $\frac{A(x,y)+A(y,x)}{2}$ for all $0\le x,y\le 1$ fixes this while not increasing $T(A)$ and not changing $S(A)$. (The fact that $T(A)$ does not increase can be verified by the above calculation dealing with the decrease of $A$ in a high-valued rectangle and the its increase in a lower valued one.)

Rearrange the intervals $I_i$ such that $\ell_1\ge \ell_2\ge\dots \ge \ell_K$. The property we just proved for the rectangles implies that for any four rectangles of the form $I_{i_1}\times I_{j_1}, I_{i_1}\times I_{j_2}, I_{i_2}\times I_{j_1}$ and $I_{i_2}\times I_{j_2}$, we have
$$A_{i_1,j_1}< A_{i_2,j_1} \Rightarrow A_{i_1,j_2}\le A_{i_2,j_2}.$$

Now we prove that $A$ is decreasing in both variables. (Since $A(x,y)=A(y,x)$, it is enough to show that for one variable.) Assume to the contrary that for some $i_1<i_2$ and $j$ we have $A_{i_1,j}<A_{i_2,j}$. Then for all $1\le p \le K$ we have $A_{i_1,p}\le A_{i_2,p}$. It results in $\ell_{i_1}<\ell_{i_2}$, a contradiction.

This decreasing property implies that if $\ell_i=\ell_{i+1}$ then $A$ is identical in $I_i\times [0,1)$ and $I_{i+1}\times [0,1)$ so we can merge all such intervals and assume that $\ell_1 > \ell_2 > \dots > \ell_{k}$ for some $k\le K$.

Note that $S(A)$ can be expressed as
\begin{equation}
\label{sa}
S(A)=\sum_{i=1}^k \sum_{j=1}^k t_i t_j \ell_i \ell_j \min(\ell_i, \ell_j).
\end{equation}

Consider an $A$ that meets the theorem's requirements, maximizes $S(A)$ and is decreasing in both variables. We state that there can not be two rectangles in the same row (or column) where the value of $A$ is neither 0 nor 1. Assume that for some $1\le a<b\le k$ and $1\le p\le k$ we have $0<A_{a,p}<1$ and $0<A_{b,p}<1$. Pick some $\epsilon\in\mathbb{R}$ and change $A_{a,p}$ and $A_{p,a}$ to $A_{a,p}+\frac{\varepsilon}{t_a t_p}$ while changing $A_{b,p}$ and $A_{p,b}$ to $A_{b,p}-\frac{\varepsilon}{t_b t_p}$. This transformation changes only two $\ell$ values: $\ell_a$ becomes $\ell_a+\frac{\varepsilon}{t_a}$ and $\ell_b$ becomes $\ell_b-\frac{\varepsilon}{t_b}$. If $|\varepsilon|$ is small enough then $0<A_{a,p}+\frac{\varepsilon}{t_a t_p}, A_{b,p}-\frac{\varepsilon}{t_b t_p}<1$ and the order of the $\ell$ values is preserved. Now we show that $S(A)$ is a strictly convex function of $\varepsilon$ in a neighborhood of 0. Consider the $k^2$ terms in the expression (\ref{sa}). The terms including other terms than $a$ and $b$ are obviously convex functions of $\varepsilon$, since $\varepsilon$ appears at a power of at most 2 in them, and it has a positive coefficient when it has power 2. So the terms of $S(A)$ corresponding to pairs of indices other than $(a,a), (a,b), (b,a)$ and $(b,b)$ are convex functions of $\varepsilon$. All we have to prove is that the sum of the terms corresponding to these four pairs is strictly convex at $\varepsilon=0$.
$$t_a^2\left(\ell_a+\frac{\varepsilon}{t_a}\right)^3+t_b^2\left(\ell_b-\frac{\varepsilon}{t_b}\right)^3+2 t_a t_b \left(\ell_a+\frac{\varepsilon}{t_a}\right)\left(\ell_b-\frac{\varepsilon}{t_b}\right)^2.$$
Differentiating twice with respect to $\varepsilon$ and substituting $\varepsilon=0$, we get
$$6\ell_a-2\ell_b+\frac{4t_a \ell_a}{t_b}.$$
It is positive, because $a<b$ implies $\ell_a>\ell_b$. Therefore $S(A)$ is a strictly convex function of $\varepsilon$ in a neighborhood of 0, so can not have a maximum at 0. This proves that the $A$ under investigation has at most one rectangle in every row and column with a value different from 0 or 1, as depicted in Figure \ref{atmost1}.

\begin{figure}[h]
\begin{center}
\includegraphics[scale = 0.3] {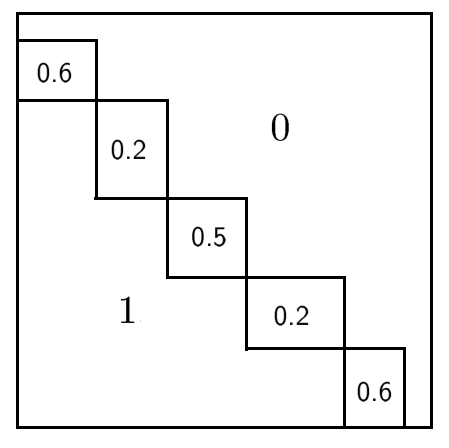}
\end{center}
\caption{Example of a function considered at this point in the proof}
\label{atmost1}
\end{figure}

Now we will prove that actually there are no such rectangles at all. Since $A$ is decreasing in both variables, each row (and column) starts with some 1-valued rectangles, then it might include a single rectangle with value between 0 and 1, then it contains 0-valued rectangles. (Of course, a row or column not necessarily contains all three types of rectangles.) If $A$ has a single rectangle of size $t\times t$ with nonzero value, then $A$'s value is $\frac{c}{t^2}$ there. (So $\frac{c}{t^2}\le 1$.) Then $S(A)=t^2\cdot (t\cdot \frac{c}{t^2})^3=\frac{c^3}{t}\le c^{\frac{5}{2}}$. If there are multiple nonzero-valued rectangles, then $A_{1,1}=1$.

Assume that $0<\lambda=A_{1,j}=A_{j,1}<1$ for some $j$. (Then $j\in\{k-1,k\}$). We will show that it is possible to modify $A$ to increase $S(A)$, so this case is not possible. (From now on we will modify the lengths of the intervals too, not just the value of $A$ in the rectangles.) Divide the interval $I_j$ into two intervals $I_{j'}$ and $I_{j''}$ of length $t_{j'}=\lambda\cdot t_j$ and $t_{j''}=(1-\lambda)\cdot t_j$ respectively. Then divide $I_1\times I_j$ into two rectangles of size $t_1\times t_{j'}$ and $t_1\times t_{j''}$, and set $A$ to be 1 and 0 respectively in them. Modify $A$ in $I_j\times I_1$ similarly to keep $A$ symmetric.

After this modification, we will get $\ell_{j'}=\frac{\ell_j}{\lambda}$ and $\ell_{j''}=0$. This means that the terms with $j''$ can be ignored in (\ref{sa}). The only terms to change in (\ref{sa}) are $t_j$ becoming $t_{j'}=\lambda\cdot t_j$ and $\ell_j$ becoming $\ell_{j'}=\frac{\ell_j}{\lambda}$. Since the power of $\ell_j$ is not smaller than the power $t_j$ in any term of (\ref{sa}), and greater than it in $t_j^2\cdot\ell_j^3$, the value of $S(A)$ increases by this modification. So we can assume that $A$ takes only 0 and 1 values in $I_1\times[0,1)$ and $[0,1)\times I_1$. To show that noninteger values are not possible in the other places, we need a technical lemma.

Note that the variables $t_i$ and $\ell_i$ appearing in the following lemma should be considered real numbers with no connection to any function $A: [0,1)^2\rightarrow [0,1]$, but the same notations are used, since the lemma will be applied in such settings.

\begin{lemma} \label{trlem}
Let $t_1, t_2, \dots, t_k$ be positive reals and let $\ell_1 > \ell_2 > \dots > \ell_{k}\ge 0$. Assume that there is a neighborhood $H$ of $t_\beta$ such that for some $\alpha\in\{\beta-1,\beta+1\}$ and $x\in H$ we can replace the numbers $t_\alpha, t_\beta, \ell_\beta$ by $t_\alpha(x)=t_\alpha+t_\beta-x$, $t_\beta(x)=x$ and $\ell_\beta(x)=\ell_\alpha+\frac{t_\beta}{x}(\ell_\beta-\ell_\alpha)$ respectively, without changing the nonnegativity of the variables and preserving the order of the $\ell$'s. The other variables are left unchanged: $t_i(x)=t_i$, if $i\not\in\{\alpha,\beta\}$ and $\ell_i(x)=\ell_i$, if $i\not=\beta$. (Roughly speaking, this transformation preserves the sum $c=\displaystyle\sum_{i=1}^k t_i(x)\ell_i(x)$, while changing only three of the values: two neighboring $t$'s and the $\ell$ corresponding to one of them.) Then the function
\begin{equation} \label{lemmaeq}
S(x)=\sum_{i=1}^k \sum_{j=1}^k t_i(x) t_j(x) \ell_i(x) \ell_j(x) \min(\ell_i(x), \ell_j(x))
\end{equation}
is strictly convex at $x=t_\beta$. Therefore it has no maximum there.
\end{lemma}

\begin{proof}
Consider the formula (\ref{lemmaeq}) and select all the terms depending on $x$. We can ignore the terms where one of the indices is $\alpha$ or $\beta$, and the other is greater than $\max(\alpha, \beta)$ because
$$(t_\alpha+t_\beta-x)t_p \ell_\alpha \ell_p^2 + x t_p \left(\ell_\alpha+\frac{t_\beta}{x}(\ell_\beta-\ell_\alpha)\right)\ell_p^2
=t_\alpha t_p \ell_\alpha \ell_p^2+t_\beta t_p \ell_\beta \ell_p^2.$$
does not depend on $x$. The sum of the other terms depending on $x$ can be written as $2S_1(x)+S_2(x)$. Here $S_1(x)$ is the sum of the terms corresponding to pairs of indices where one of the elements is $\alpha$ or $\beta$ and the other one in smaller than $\alpha$ and $\beta$. $S_2(x)$ denotes the sum of the terms corresponding to the pairs of indices $(\alpha, \alpha), (\alpha, \beta), (\beta, \alpha)$ and $(\beta, \beta)$.
$$S_1(x)=\left((t_\alpha+t_\beta-x)\cdot \ell_\alpha^2 +
x\cdot \left(\ell_\alpha+\frac{t_\beta}{x}(\ell_\beta-\ell_\alpha)\right)^2\right)\cdot\sum_{p=1}^{\min(\alpha,\beta)-1} t_p\ell_p,$$

$$S_2(x)=(t_\alpha+t_\beta-x)^2\cdot \ell_\alpha^3+x^2\cdot \left(\ell_\alpha+\frac{t_\beta}{x}(\ell_\beta-\ell_\alpha)\right)^3+$$
$$2(t_\alpha+t_\beta-x)x \ell_\alpha \left(\ell_\alpha+\frac{t_\beta}{x}(\ell_\beta-\ell_\alpha)\right) \min\left(\ell_\alpha, \ell_\alpha+\frac{t_\beta}{x}(\ell_\beta-\ell_\alpha)\right).$$

We have to show that $S_1(x)$ and $S_2(x)$ are strictly convex at $x=t_\beta$, therefore $S(x)$ does not takes its maximum there. We will start with $S_1(x)$. We can disregard the constant factor at the right, as it does not change convexity. The left factor can be expressed as $\lambda_2 x^2+\lambda_1 x+\lambda_0+\lambda_{-1} x^{-1}$. Since $\lambda_2=\ell_\alpha^2>0$ and $\lambda_{-1}=t_\beta(\ell_\beta-\ell_\alpha)^2>0$, $S_1$ is strictly convex.

Now we consider $S_2(x)$. First, assume that $\beta=\alpha+1$, and therefore $\ell_\alpha>\ell_\beta$. In this case we have
$$S_2(x)=(t_\alpha+t_\beta-x)^2\cdot \ell_\alpha^3+x^2\cdot \left(\ell_\alpha+\frac{t_\beta}{x}(\ell_\beta-\ell_\alpha)\right)^3+
2(t_\alpha+t_\beta-x)x \ell_\alpha \left(\ell_\alpha+\frac{t_\beta}{x}(\ell_\beta-\ell_\alpha)\right)^2.$$
Differentiating twice by $x$ and setting $x=t_\beta$ we get
$$2t_\beta^{-1}(\ell_\beta-\ell_\alpha)^2\big((\ell_\alpha+\ell_\beta)t_\beta+2\ell_\alpha t_\alpha\big)>0.$$
If $\beta=\alpha-1$, and therefore $\ell_\alpha<\ell_\beta$, we have
$$S_2(x)=(t_\alpha+t_\beta-x)^2\cdot \ell_\alpha^3+x^2\cdot \left(\ell_\alpha+\frac{t_\beta}{x}(\ell_\beta-\ell_\alpha)\right)^3+
2(t_\alpha+t_\beta-x)x \ell_\alpha^2 \left(\ell_\alpha+\frac{t_\beta}{x}(\ell_\beta-\ell_\alpha)\right).$$
Differentiating twice by $x$ and setting $x=t_\beta$ we get
$$2(\ell_\beta-\ell_\alpha)^3>0.$$
In both cases, $S_2(x)$ is strictly convex at $x=t_\beta$. This concludes the proof of the lemma.
\end{proof}

Now we can continue the proof of the theorem. Assume that $A$ is not entirely 0-1 valued. Let $I_p\times [0,1)$ be the first column containing a rectangle with a value different from 0 and 1. We already proved that $p\ge 2$. Let $I_p\times I_q$ the unique rectangle in the $p$-th column with $0<A_{p,q}<1$. Since $A_{p,q}=A_{q,p}$, we know that $p\le q$. We will show that $A$ admits the type of transformation described in Lemma \ref{trlem}, therefore does not maximize $S(A)$. We will describe transformations in each case that change only two neighboring $t$ values and the $\ell$ corresponding to one of them.

{\it Case a}: First, assume that $p<q<k$ and $A_{p-1,q+1}=1$. Then the rows $[0,1)\times I_q$ and $[0,1)\times I_{q+1}$ differ only in the $p$-th column. We can move the point separating the intervals $I_q$ and $I_{q+1}$, keeping $A_{p,q+1}=0$ while adjusting $A_{p,q}$ such that the integral of $A$ over the whole square remains unchanged. We apply these changes to the other side of the main diagonal to keep $A$ symmetric. During this transformation the only $\ell$ value to change is $\ell_q$. So Lemma \ref{trlem} can be applied with $\alpha=q+1, \beta=q$.

\begin{figure}[h]
\begin{center}
\includegraphics[scale = 0.5] {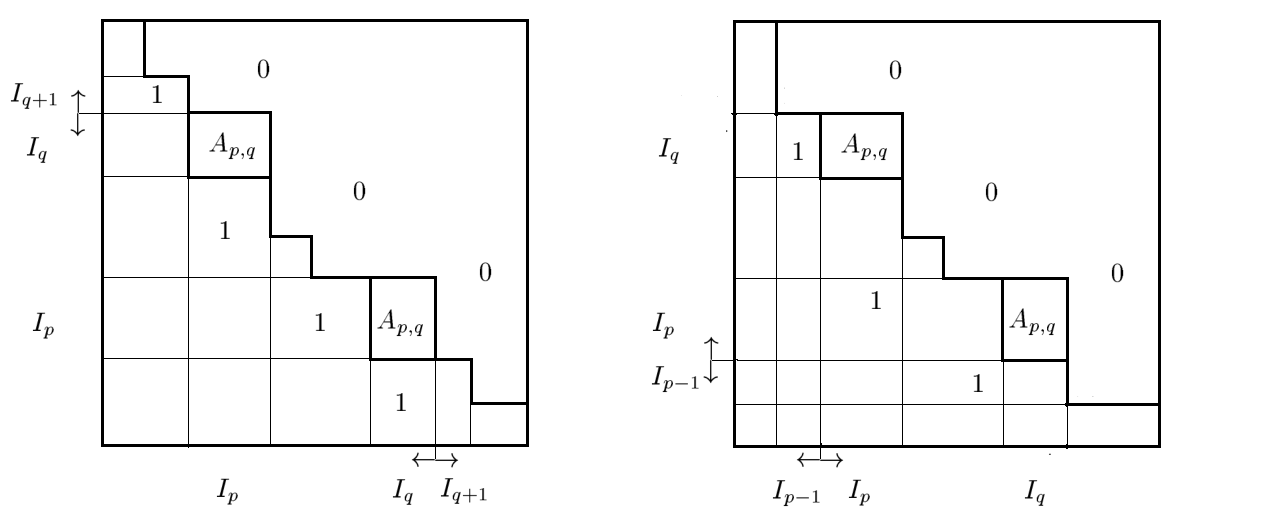}
\end{center}
\caption{Illustration for Case a (left) and Case b (right).}
\end{figure}

{\it Case b}: Now assume that $p<q=k$ or $p<q<k$ and $A_{p-1, q+1}=0$. Then the columns $I_{p-1}\times [0,1)$ and $I_{p}\times [0,1)$ differ only in the $q$-th row. We can move the point separating the intervals $I_{p-1}$ and $I_{p}$, keeping $A_{p-1,q}=1$ while adjusting $A_{p,q}$ such that the integral of $A$ over the whole square stays the same. We apply these changes to the other side of the main diagonal to keep $A$ symmetric. During this transformation the only $\ell$ value to change is $\ell_p$. So Lemma \ref{trlem} can be applied with $\alpha=p-1, \beta=p$.

\begin{figure}[h]
\begin{center}
\includegraphics[scale = 0.5] {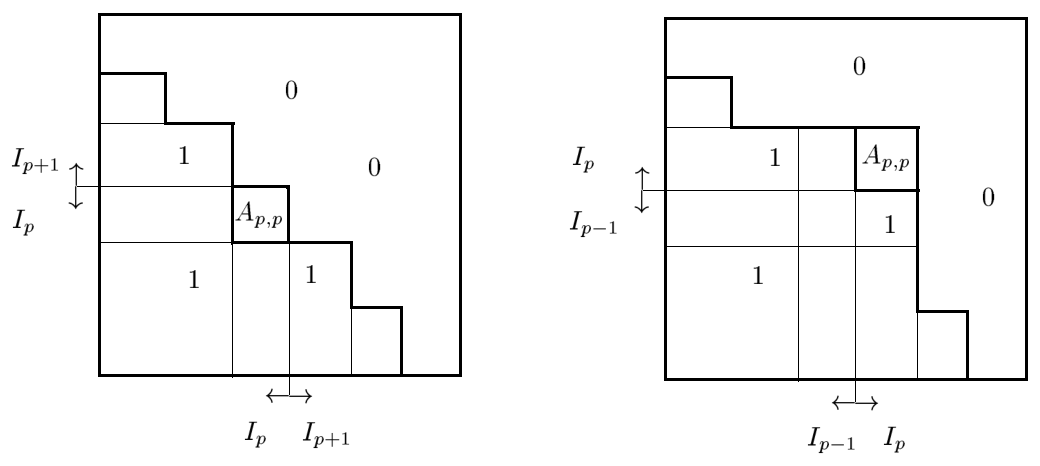}
\caption{Illustration for Case c (left) and Case d (right).}
\end{center}
\end{figure}

{\it Case c}: Assume that $p=q<k$ and $A_{p-1,p+1}=1$. Then the rows $[0,1)\times I_p$ and $[0,1)\times I_{p+1}$ differ only in the $p$-th column. We can move the point separating the intervals $I_p$ and $I_{p+1}$, keeping $A_{p,p+1}=0$ while adjusting $A_{p,p}$ such that the integral of $A$ over the whole square stays the same. (We apply these changes to the the intervals defining the rows and columns simultaneously to keep $A$ symmetric.) During this transformation the only $\ell$ value to change is $\ell_p$. So Lemma \ref{trlem} can be applied with $\alpha=p+1, \beta=p$.

{\it Case d}: Now assume that $p=q=k$ or $p=q<k$ and $A_{p-1, p+1}=0$. Then the columns $I_{p-1}\times [0,1)$ and $I_{p}\times [0,1)$ differ only in the $p$-th row. We can move the point separating the intervals $I_{p-1}$ and $I_{p}$, keeping $A_{p-1,p}=1$ while adjusting $A_{p,p}$ such that the integral of $A$ over the whole square stays the same. (We apply these changes to the the intervals defining the rows and columns simultaneously to keep $A$ symmetric.) During this transformation the only $\ell$ value to change is $\ell_p$. So Lemma \ref{trlem} can be applied with $\alpha=p-1, \beta=p$.

With this, we have covered all the possibilities. From now on, we can assume that $A: [0,1)^2\rightarrow \{0,1\}$.

Since $A$ is 0-1 valued and decreasing in both variables, there exists some $k'\in \{k-1,k\}$ such that $k$ of the $\ell$ values is positive and $A_{i,j}=1$ if and only if $i+j\le k'+1$. Now we will show that $A$ can not maximize $S(A)$ if $k'\ge 4$.

Assume that $k'\ge 4$. It is possible to move the point separating the intervals $I_{k'-1}$ and $I_{k'}$, while keeping $\ell_{k'-1}$ unchanged and adjusting $\ell_{k'}$ to keep the integral over the whole square unchanged. When these changes are applied to the other side of the main diagonal to preserve the symmetry, we se that the point separating $I_1$ and $I_2$ moves, $\ell_1$ remains unchanged, but $\ell_2$ changes. During this transformation, the following values change (see Figure \ref{k4}):

$t_{k'}\rightarrow x,$

$t_{k'-1}\rightarrow t_{k'-1}+t_{k'}-x,$

$\ell_{k'}\rightarrow \ell_{k'-1}+\frac{t_{k'}}{x}(\ell_{k'}-\ell_{k'-1}),$

$\ell_2\rightarrow\ell_1-x,$

$t_2\rightarrow \frac{t_2(\ell_1-\ell_2)}{x},$

$t_1\rightarrow t_1+t_2-\frac{t_2(\ell_1-\ell_2)}{x}.$

\begin{figure}[h]
\begin{center}
\includegraphics[scale = 0.4] {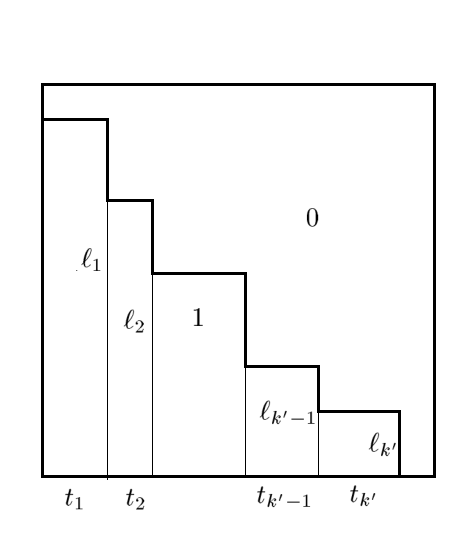}
\caption{The case $k'\ge 4$}
\label{k4}
\end{center}
\end{figure}

Now we investigate how $S(A)$ changes during such a transformation. First, apply the changes only to $t_{k'}, t_{k'-1}$ and $\ell_{k'}$. Lemma \ref{trlem} states that $S$ is now a strictly convex function of $x$. (Because we changed only two neighboring $t$'s and the $\ell$ corresponding to one of them, while preserving the sum $\displaystyle\sum_{i=1}^{k'} t_i\ell_i$.) Now apply the changes to $t_1, t_2$ and $\ell_2$ too. Since $t_1\ell_1+t_2\ell_2$ does not change during the transformation, for any $3\le s \le k'$, the sum of the terms in (\ref{sa}) corresponding to the pairs of indices (1,$s$), ($s$,1), (2,$s$) and ($s$,2) which is
$$2t_1t_s\ell_1\ell_s^2+2t_2t_s\ell_2\ell_s^2=2t_s\ell_s^2(t_1\ell_1+t_2\ell_2),$$
does not change. Therefore it is enough to consider the terms where both indices are 1 or 2.
$$\left(t_1+t_2-\frac{t_2(\ell_1-\ell_2)}{x}\right)^2\ell_1^3+\left(\frac{t_2(\ell_1-\ell_2)}{x}\right)^2(\ell_1-x)^3+
2\left(t_1+t_2-\frac{t_2(\ell_1-\ell_2)}{x}\right)\left(\frac{t_2(\ell_1-\ell_2)}{x}\right)\ell_1(\ell_1-x)^2.$$ 
Differentiating two times by $x$ and setting $x=t_{k'}=\ell_1-\ell_2$ we get $\frac{2t_2^2 \ell_1^2}{t_{k'}}>0$, so the above formula is a strictly convex function of $x$. As the sum of two strictly convex functions, $S$ is a strictly convex function of $x$, therefore $A$ does not maximize $S(A)$.

Now assume that $k'=3$. We state that if $A: [0,1)^2\rightarrow \{0,1\}$ is a symmetric function decreasing in both variables with at most three positive $\ell$ values, then there is another such function $B$ with at most two positive $\ell$ values such that $\int_0^1 \int_0^1 A(x,y) \, \mathrm{d}x \mathrm{d}y=\int_0^1 \int_0^1 B(x,y) \, \mathrm{d}x \mathrm{d}y$ and $S(B)\ge S(A)$.
If $A$ is such a function then $\ell_1=t_1+t_2+t_3$, $\ell_2=t_1+t_2$ and $\ell_3=t_1$. Without loss of generality we can assume that $t_1+t_2+t_3=\ell_1=1$. (Replacing $A(x,y)$ with $A(\lambda x, \lambda y)$ changes $S(A)$ with a factor of $\lambda^5$, so the rescaling does not change our problem.) Note that in this case the two parameters $s=1-\int_0^1 \int_0^1 A(x,y) \, \mathrm{d}x \mathrm{d}y$ and $x=t_3$ are enough two define $A$. (See Figure \ref{k3}.) We have

$\ell_1=1,$

$\ell_2=1-x,$

$t_1=\ell_3=1-\frac{s+x^2}{2x},$

$t_2=\frac{s-x^2}{2x},$

$t_3=x.$

\begin{figure}[h]
\begin{center}
\includegraphics[scale = 0.4] {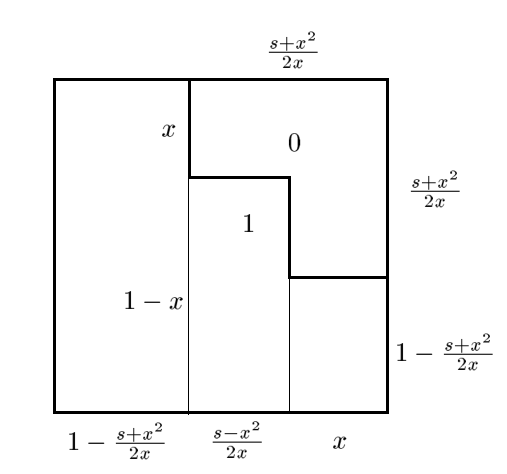}
\caption{The case $k'=3$}
\label{k3}
\end{center}
\end{figure}

Note that $x$ can take any value from $[1-\sqrt{1-s},\sqrt{s}]$. The two endpoints correspond to functions with at most two intervals. (After scaling back, we get a step function with at most two positive $\ell$ values.)

$$S(A)=f(x)=t_1^2\ell_1^3+t_2^2\ell_2^3+t_3^2\ell_3^3+2t_1t_2\ell_1\ell_2^2+2t_1t_3\ell_1\ell_3^2+2t_2t_3\ell_2\ell_3^2=$$
$$\left(1-\frac{s+x^2}{2x}\right)^2+\left(\frac{s-x^2}{2x}\right)^2(1-x)^3+x^2\left(1-\frac{s+x^2}{2x}\right)^3+$$
$$2\left(1-\frac{s+x^2}{2x}\right)\left(\frac{s-x^2}{2x}\right)(1-x)^2+
2\left(1-\frac{s+x^2}{2x}\right)x\left(1-\frac{s+x^2}{2x}\right)^2+
2\left(\frac{s-x^2}{2x}\right)x(1-x)\left(1-\frac{s+x^2}{2x}\right)^2.$$

We will prove that for a fixed $s$, $f(x)$ is either increasing in $[1-\sqrt{1-s},\sqrt{s}]$ or there exists an $x_0\in [1-\sqrt{1-s},\sqrt{s}]$ such that $f(x)$ is strictly decreasing in $[1-\sqrt{1-s},x_0]$ and strictly increasing in $[x_0,\sqrt{s}]$. In both cases, $f(x)$ must take its maximum in one of the endpoints. Differentiate $f(x)$ by $x$. We need that $f'(x)$ is either positive in $(1-\sqrt{1-s},\sqrt{s})$ or it is negative in $(1-\sqrt{1-s},x_0)$ and positive in $(x_0,\sqrt{s})$. Since $x>0$, it is sufficient to prove the same for $f'(x)x$. An elementary calculation shows that $\lim_{x\searrow 0} f'(x)x=-\infty$ and $f'(\sqrt{s})\sqrt{s}=0$. If we could show that $f'(x)x$ is strictly concave in $[0,\sqrt{s}]$, then the desired result would follow. After further calculation
$$4(f'(x)x)''=-50x^3+96x^2+(27s-54)x-16s-3s^2(2-s)\frac{1}{x^3}.$$

We are going to prove that the above formula is negative if $0<x,s<1$. Since $s<1$ it enough to show that
$$g(x,s)=-50x^3+96x^2+(27s-54)x-16s-3s^2\frac{1}{x^3}\leq 0.$$
This is a polynomial of $s$ of degree 2. For a fixed $x$, it takes its maximum at $s=\frac{27x^4-16x^3}{6}$.

If $\frac{27x^4-16x^3}{6}\ge 1$, then $x\ge 0.89$ and
$$g(x,s)\le g(x,1)=-50x^3+96x^2-27x-16-3x^{-3}\le 0.$$

If $\frac{27x^4-16x^3}{6}\le 1$, then $x\le 0.9$ and
$$g(x,s)\le g\left(x,\frac{27x^4-16x^3}{6}\right)=\frac{1}{12}(729x^5-864x^4-344x^3+1152x^2-648x)\le 0.$$

(Both of the above results follow by elementary calculus.) This concludes the proof of the case $k'=3$. We obtained that $S(A)$ is maximized by a function with at most two positive $\ell$ values.

Now we can assume that $k'\le 2$. Then $A$ is completely defined by the parameters $\int_0^1 \int_0^1 A(x,y) \, \mathrm{d}x \mathrm{d}y=c$ and $t_1=x$. (See Figure \ref{k2}.)

$t_1=\ell_2=x,$

$t_2=\frac{c-x^2}{2x},$

$\ell_1=\frac{c+x^2}{2x}.$

\begin{figure}[h]
\begin{center}
\includegraphics[scale = 0.4] {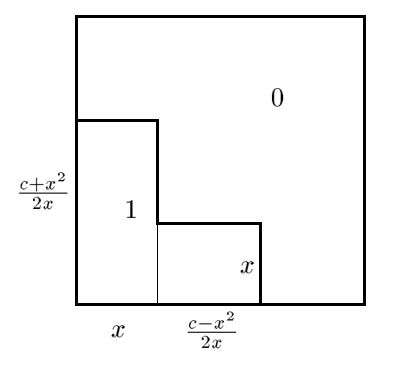}
\caption{The case $k'=2$}
\label{k2}
\end{center}
\end{figure}

Note that $1-\sqrt{1-c} \le x\le \sqrt{c}$, and $x=1-\sqrt{1-c}$ corresponds to $A_1(c)$, while $x=\sqrt{c}$ corresponds to $A_2(c)$. (See Figure \ref{a1a2}.)

$$S(A)=t_1^2\ell_1^3+t_2^2\ell_2^3+2t_1t_2\ell_1\ell_2^2=x^2 \left(\frac{c+x^2}{2x}\right)^3+ \left(\frac{c-x^2}{2x}\right)^2 x^3 + 2x\left(\frac{c-x^2}{2x}\right)\left(\frac{c+x^2}{2x}\right)x^2=$$

$$\frac{1}{8}\left(\frac{c^3}{x}+9c^2x-cx^3-x^5\right).$$

Using the substitution $y=\frac{x}{\sqrt{c}}$ (where $\frac{1-\sqrt{1-c}}{\sqrt{c}} \le y\le 1$) we get
$$S(A)=\frac{c^{\frac{5}{2}}}{8}\left(\frac{1}{y}+9y-y^3-y^5\right).$$

We want to show that this function takes its maximum at one of the endpoints of its domain. It suffices to show that there exists a real number $0<y_0<1$ such that the function $y\rightarrow \frac{1}{y}+9y-y^3-y^5$ is strictly decreasing in $(0,y_0)$ and strictly increasing in $(y_0,1)$.

Differentiating once we get the function $f(y)=-\frac{1}{y^2}+9-3y^2-5y^4$. We need that there is some $0<y_0<1$ such that $f(y)<0$ if $0<y<y_0$ and $f(y)>0$ if $y_0<y<1$. Consider the function $g(y)=f(\sqrt{y})y$. It is obvious that $g$ has the desired property if and only $f$ has it. Since $g(y)=-5y^3-3y^2+9y-1$, a polynomial of degree 3, this property can be verified for $g$ by elementary calculus. With this, Theorem \ref{thm2} is proved.
\end{proof}

With this, we proved that $\frac{N'}{n^5}\le S(A_G)\le \max(A_1(c), A_2(c))$, finishing the proof of the upper bound. Using the formula (\ref{sa}), we find that $A_1(c)=(1-\sqrt{1-c})^2((c+1)\sqrt{1-c}+c)$ and $A_2(c)=c^{5/2}$. This concludes the proof of Theorem \ref{main}. By plotting these two functions we can conclude that there is some $c_0\approx 0.0865$ such that $A_1(c)\ge A_2(c)$ in $[0,c_0]$ and $A_1(c)\le A_2(c)$ in $[c_0,1]$.
\end{proof}

\section{Remarks and open questions}

First, we note that the bounds in Theorem \ref{main} are asymptotically sharp.

\begin{remark}
Let $n$ and $e$ be fixed positive integers satisfying $e\le \binom{n}{2}$. Let $c=\frac{2e}{n^2}$. Then there is a simple graph $G_1$ with $n$ vertices and $e$ edges containing at most $\frac{1}{2}c^4n^5$ 4-edge paths. Additionally, there are simple graphs $G_2$ and $G_3$ with $n$ vertices and $e$ edges that contain at least $\frac{1}{2}c^\frac{5}{2}n^5(1-O(n^{-1}))$ and $\frac{1}{2}(1-\sqrt{1-c})^2((c+1)\sqrt{1-c}+c)n^5(1-O(n^{-1}))$ 4-edge paths respectively.
\end{remark}

\begin{proof}
Let $G_1$ be a graph with $n$ vertices and $e$ edges such that the degree of any two vertices differ by at most 1. (It is well-known that such a graph exists.) Then the degree of any vertex is at most $\frac{2e}{n}+1=cn+1$. So the number of 4-edge paths is at most
$$\frac{1}{2}n(cn+1)(cn)(cn-1)(cn-2)\le\frac{1}{2}n(cn)^4=\frac{1}{2}c^4n^5.$$

Now we show that we can choose the quasi-clique $C_n^e$ for $G_2$. $C_n^e$ contains an $a$-clique, where $a$ is the greatest integer satisfying $\binom{a}{2}\le e$. Therefore $\binom{a+1}{2}\ge e$, implying $a\ge \sqrt{2e}-1$. The number of 4-edge paths in this clique is
$$\frac{1}{2}a(a-1)(a-2)(a-3)(a-4)\ge \frac{1}{2}(2e)^\frac{5}{2}(1-O(e^{-\frac{1}{2}}))\ge \frac{1}{2}c^\frac{5}{2}n^5(1-O(n^{-1})).$$

A similar (but more complicated) calculation gives that we can choose the quasi-star $S_n^e$ for $G_3$.
\end{proof}

We conclude the paper with a few open questions.

\begin{question}
We proved that either the quasi-star or the quasi-clique asymptotically maximizes the number of 4-edge paths in graphs with given edge density. Is it true that this maximum is actually exactly (not just asymptotically) achieved by either the quasi-star or the quasi-clique?
\end{question}

Theorem \ref{ahlkat} states that the above is true for 2-edge paths.

\begin{question}
Is it true for all graphs $H$ that the number of subgraphs isomorphic to $H$ in graphs with given edge density is (asymptotically) maximized by either the quasi-star or the quasi-clique?
\end{question}

It is true for 4-edge paths and $k$-edge stars, when $2\le k \le 30$ (see Theorem \ref{starthm}). When $H$ is a graph having a spanning subgraph that is a vertex-disjoint union of edges and cycles, only the quasi-clique comes into play (see Theorem \ref{alonthm}).

\begin{question}
Is it true that for every graph $H$, there is a constant $c_H<1$ such that among graphs with $n$ vertices and edge density $c>c_H$, the number of subgraphs isomorphic to $H$ is (asymptotically) maximized by the quasi-clique?
\end{question}

\noindent
{\bf Acknowledgement} I would like to thank Gyula O.H. Katona for his help with the creation of this paper.

\end{document}